\documentclass[11pt,reqno]{amsproc}
\usepackage[margin=1in]{geometry}
\usepackage{amsmath, amsthm, amssymb}
\usepackage{hyperref}
\hypersetup{colorlinks=true, pdfstartview=FitV, linkcolor=blue,citecolor=blue, urlcolor=blue}
\usepackage[abbrev,lite,nobysame]{amsrefs}
\usepackage{times}
\usepackage[usenames,dvipsnames]{color}
\usepackage{bbm}
\usepackage{bm}
\usepackage{subcaption}
\usepackage{mathtools}
\usepackage{pdfpages}
\usepackage{comment}
\definecolor{labelkey}{rgb}{0,0,1}

\newcommand{\R}{\mathbb{R}}

\newcommand{\N}{\mathbb{N}}

\def\rn{{\mathbb R}^N}
\def\rnp{{\mathbb R}^N_+}

\def\crnp{\overline{\mathbb R}^N_+}

\def\comega{\overline\Omega }
\def\ang#1{\langle {#1} \rangle}

\def\N{\mathbb N}
\def\R{\mathbb R}

\providecommand{\R}{\mathbb{R}}

\providecommand{\N}{\mathbb{N}}

\newcommand{\pv}{\operatorname{p.\!v.}}

\renewcommand{\leq}{\leqslant}

  \newtheorem{thm}{Theorem}[section]
  
  \newtheorem{Lemma}[thm]{Lemma}
  \newtheorem{Proposition}[thm]{Proposition}
    \newtheorem{Definition}[thm]{Definition}
  \newtheorem{remark}[thm]{Remark}

\usepackage{etoolbox}
\patchcmd{\subsubsection}{\itshape}{\itshape\bfseries}{}{} 

\title[]{A Lions' type formula for some reproducing kernel Hilbert spaces of fractional harmonic functions}

\author[]{Sidy M. Djitte and Franck Sueur}

\address[S. M. Djitte]{Department of Mathematics
Maison du nombre, 6 avenue de la Fonte, 
University of Luxembourg,  
L-4364 Esch-sur-Alzette, Luxembourg}
\email{sidymoctar.djitte@uni.lu}
\address[F. Sueur]{Department of Mathematics
Maison du nombre, 6 avenue de la Fonte, 
University of Luxembourg,  
L-4364 Esch-sur-Alzette, Luxembourg}
\email{Franck.Sueur@uni.lu}

\keywords{Reproducing kernel Hilbert space, Hadamard variation formula, $a$-transmission Sobolev spaces}

\begin{document}


\begin{abstract}
In \cite{Lions}, J. L. Lions considered a reproducing kernel Hilbert space (RKHS) of harmonic functions on a regular domain with Sobolev traces and obtained a formula that expresses the kernel of this space as an integral on the boundary of some derivatives of the Green function  associated with the  Laplace operator and the  homogeneous Dirichlet boundary condition. 
This result was simplified and extended later by Englis,  Lukkassen,  Peetre, and  Persson in \cite{ELPL} to more general elliptic systems of even orders. In particular, they emphasized that the resemblance between Lions' type formula and the  Hadamard variational formula only appears when the operator is of order $2$. 
In this paper, we  investigate some RKHS of $a$-harmonic functions, where $a$ in $(0,1)$ refers to a fractional exponent of the Laplace operator. For such fractional order pseudo-differential operators,  the local nonhomogeneous Dirichlet problem can be addressed by means of some $a$-transmission Sobolev spaces, which were introduced by H\"ormander in the sixties and recently developed by Grubb in a series of papers. We deduce from these works a fractional Poisson formula, which is applied to obtain a Lions' type formula. We observe, in particular, that despite the order of the operator not being $2$, this formula resembles the  Hadamard variational formula that we prove in the companion paper \cite{sidy-franck_1}.
As a complementary remark, we observe that for a family of RKHS associated with 
the steady Stokes system, a second order system, 
 there is also a  Lions' type formula for their two-point kernels, which turn out not to be similar to the corresponding Hadamard variation formula.
\end{abstract}

\maketitle

\setcounter{tocdepth}{3}
\tableofcontents

\section{Introduction}
\subsection{Reproducing kernel Hilbert space}

The notion of  reproducing kernel Hilbert space was first introduced in 1907 by Stanislaw Zaremba for boundary value problems for harmonic and biharmonic functions, and simultaneously by James Mercer in the theory of integral equations, before being more systematically tackled by Nachman Aronszajn and Stefan Bergman. These spaces have various applications in complex analysis, harmonic analysis, quantum mechanics, and statistical learning theory.
They are now defined as Hilbert spaces of functions in which pointwise evaluations are continuous linear functionals. It then follows from Riesz' theorem that each of these functionals can be represented as an inner product with an element of this Hilbert space. One defines a two-point kernel by considering the inner products of any pair of such elements. \ \par \ 
More precisely, we have the following definition, where  $\Omega$, as in all the paper, is a bounded connected open set of class $C^\infty$ in $\R^N$.
\begin{Definition} \label{defRKHS}
We say that a Hilbert space $H$ of  real-valued functions defined on  $\Omega$ is a reproducing kernel Hilbert space (RKHS for short) 
 if for any $x$ in $\Omega$, the evaluation mapping $E_x$ defined for any $u$ in $H$ by $E_x (u) = u(x)$ 
is linear and continuous from $H$ to $\R$. Then, we denote by  $K_x$ the unique element in $H$ such that for any $u$ in $H$,  $E_x (u) = \langle u,K_x\rangle $, the inner product of $u$ and $K_x$ in $H$, given by Riesz' theorem. Finally, the two-point kernel is defined for any pair of $x$ and $y$ in $\Omega$  by $K(x,y)= \langle K_x,K_y\rangle $. 
\end{Definition}
Observe that it follows from the properties of the inner product that the kernel $K$ is symmetric and positive definite. 
Actually, a famous theorem by  E. H. Moore and N. Aronszajn, see \cite{Aronszajn},  states that every symmetric, positive definite kernel defines a unique reproducing kernel Hilbert space. 
%

\subsection{Lions' formula for a reproducing kernel Hilbert space of harmonic functions}
In \cite{Lions}, J.L. Lions has established an interesting formula for certain reproducing kernel Hilbert spaces of harmonic functions. 
To recall his result, let us introduce the following differential operator on functions defined on the boundary $\partial \Omega$. 
\begin{Definition} \label{DefBel}
Let $L$ be the negative of the Laplace-Beltrami operator on $\partial \Omega$, and let $M  := L+ 1$. This operator $M$ is called  the  smoothed Laplace-Beltrami operator on $\partial \Omega$.  
\end{Definition}
Let $L^2 (\partial \Omega)$ the Hilbert space of square integrable functions on $\partial \Omega$
with respect to the $d-1$-dimensional measure $d\sigma$ on $\partial \Omega$. 
Then the  smoothed Laplace-Beltrami operator $M$ is a positive self-adjoint operator in $L^2 (\partial \Omega)$, so that, by Borel  functional calculus, it makes sense to speak of its powers of  any real order. This allows us to define some Sobolev spaces on $\partial \Omega$ for any $s$ in $\R$ as follows.
\begin{Definition}\label{def1.3}
For any $s$ in $\R$,  the Sobolev space $H^s(\partial \Omega)$ is the closure of the smooth functions on $\partial \Omega$ for the norm associated with the scalar product
\begin{equation} \label{inpr}
\langle u,v\rangle_s := 
\big\langle M^{\frac{s}{2}} \, u,M^{\frac{s}{2}} \,\,v\big\rangle_{L^2 (\partial \Omega)} .
\end{equation}
\end{Definition}
The Dirichlet trace $\gamma_D$ on the boundary $\partial \Omega$, defined by 
\begin{equation} \label{gammad}
    \gamma_D: \, u \in C^\infty (\overline \Omega) \mapsto  u\vert_{\partial \Omega} \in C^\infty (\partial \Omega)
\end{equation} 
extends 
in a continuous linear map from the  Sobolev space $H^s(\Omega)$
to  the Sobolev space $H^{s-\frac12} (\partial \Omega)$ for all $s > \frac12$. It does not extend to the negative Sobolev spaces $H^{t} (\partial \Omega)$, with $t\leq 0$; with such properties. 
However, it extends to the subspace of harmonic functions of the Sobolev space $H^s(\Omega)$ for any $s \in \R$, again as a continuous linear map  with values in  $H^{s-\frac12} (\partial \Omega)$. 
Actually, this holds not only for harmonic functions but also for much more general 
solutions of elliptic equations, as established by the works of  Lions and Magenes, see \cite{LM63,LM68}, by establishing some 
 regularity properties in the case of Sobolev spaces with positive indexes first, then by using a duality argument and finally interpolation theory.
 Another approach, based on the theory of pseudodifferential operators for boundary value problems,  
 emphasized that such a property can be seen as a variant of the partial hypoellipticity of the elliptic equations and originated from H\"ormander in \cite[Theorem 2.5.6]{H63} and Boutet de Monvel \cite{B71}, see also  \cite[Theorem 11.4]{Gr} for a more recent account of the topic. 
With a slight abuse of notation, we keep the same notation $\gamma_D$ for this trace map, as a continuous linear map from the subspace of harmonic functions of the Sobolev space $H^s(\Omega)$ to $H^{s-\frac12} (\partial \Omega)$ 
for any $s \in \R$.

Now, the space of harmonic distributions considered by Lions is the following.
\begin{Definition} \label{def-Hs}
For any $s \in \R$, we define $\mathcal H_s(\Omega) $ as the space of harmonic functions $u$ in $ H^{s+\frac12} (\Omega)$  whose  trace $\gamma_D (u)$  on $\partial \Omega$ belongs to $H^s(\partial \Omega)$.
\end{Definition}
It follows from the references above that for any 
$s$ in $\R$,  any $u$ in  $\mathcal H_s (\Omega)$ is actually a $C^\infty$ function in $\Omega$, and for any $x$ in $\Omega$, 
 we have the Poisson formula: 
%
\begin{equation} \label{Cani1}
    u(x)
    = \Big\langle \gamma_D (u),\, M^{-{s}} \,\gamma_N \big( G_1 (x,\cdot)\big) \Big\rangle_s ,
\end{equation}
where $G_1$ is the Green function, that is the function such that for any $x$ in $\Omega$, 
$G_1(x,\cdot)$  is the classical Green function with singularity at $x$, that is the solution to 
\begin{equation} \label{DIR}
    -\Delta G_1(x,\cdot)=\delta_x \quad\text{in}\quad \mathcal D'(\Omega)\qquad\&\qquad G_1(x,\cdot)=0\quad\text{on}\quad \partial\Omega, 
\end{equation}
and where $\gamma_N \big(G_1(x,\cdot)\big)$ denotes its Neumann trace, that is 
$$\gamma_N \big(G_1(x,\cdot)\big)(z):= \nabla_z G_1(x,z)\cdot\nu(z) \quad \text{for}\quad  z\in\partial\Omega,$$ where $\nu(z)$ is the outer unit normal at $z$.
In \eqref{DIR}, $\delta_x$ denotes the Dirac delta distribution at the position $x$. 
In particular, it follows from elliptic regularity that 
\begin{equation} \label{REg}
    \forall x \in \Omega, \quad   \forall s \in \R, \quad M^{-\frac{s}{2}} \gamma_N \big(G_1(x,\cdot)\big)  \in L^2 (\partial\Omega) ,
\end{equation}
 so that  the 
right-hand side of \eqref{Cani1} 
makes sense. Moreover, for any $s$ in $\R$, the mapping that associates any $g$ in  $H^s(\partial \Omega)$ to the function $\mathcal P_s [g]$ defined for any  $x$  in $\Omega$, by 
\begin{equation} \label{poisson-iso}
 \mathcal P_s[g] (x) :=   \Big\langle g, \,M^{-{s}} \,\gamma_N \big( G_1 (x,\cdot)\big) \Big\rangle_s ,
\end{equation}
is a one-to-one isometry onto $\mathcal H_s(\Omega) $. This allows us to transport the Hilbert structure from $H^s(\partial \Omega)$ to $\mathcal H_s(\Omega) $, that is we 
 equip $\mathcal H_s(\Omega) $  with the scalar product, defined for any $ u,v$ in $\mathcal H_s(\Omega) $, by 
 $\langle \gamma_D (u),\gamma_D (v) \rangle_s $. 
Actually, Lions' result, see \cite[Formula (2.13)]{Lions}, is that it is a RKHS with a  two-point kernel given as a boundary integral involving the 
Green function $G_1$.

\begin{thm} \label{Lions-theo}
For any $s$ in $\R$, the space $\mathcal H_s(\Omega) $  is a RKHS and its two-point kernel $K_s$ is given,  for any pair of $x$ and $y$ in $\Omega$,  by 
 \begin{equation} \label{l-formu}
 K_s (x,y)=  \int_{\partial\Omega} \Big((M^{-\frac{s}{2}} \, \gamma_N \big((G_1(x,\cdot)\big)\Big)(z) \, \Big(M^{-\frac{s}{2}} \, \gamma_N \big((G_1(y,\cdot)\big) \Big)(z)\,d\sigma(z) .
 \end{equation}
\end{thm}
  Let us point out that the integral in \eqref{l-formu} makes sense, and is finite, since  for any  $x$  in $\Omega$, for any $s$ in $\R$, 
$M^{-\frac{s}{2}} \gamma_N \big((G_1(x,\cdot)\big)   $
is in $L^2 (\partial\Omega)$.
Actually, Lions dealt with some function spaces associated with second-order elliptic boundary value problems, see \cite{Lions}.  
Lions used a variational approach and obtained his formula through a penalization limit. 
Another, more direct approach to Theorem \ref{Lions-theo} was proposed later in \cite{ELPL}, together with some extensions to more general elliptic operators of even orders. 
We reproduce below 
their proof of \eqref{l-formu} as a preparation for our extension to the case of fractional Laplace operators.

 %
 \begin{proof}
 It follows from the Poisson formula \eqref{Cani1}, \eqref{REg}
and the Cauchy-Schwarz inequality that for any $x$ in $\Omega$, the evaluation mapping $E_x$, see Definition \ref{defRKHS},   is continuous from $\mathcal H_s(\Omega) $ to $\R$. 
Since $\mathcal H_s(\Omega) $ is a Hilbert space, 
 it follows from Riesz' theorem that  for any  $x$  in $\Omega$, there exists $K_x$ in 
$\mathcal H_s(\Omega) $ 
such that for any $u$ in $\mathcal H_s(\Omega) $, 
\begin{equation} \label{Cani2}
u(x) = \langle \gamma_D (u) ,\, K_x \rangle_s = \langle M^{\frac{s}{2}} \, \gamma_D (u) ,\,  M^{\frac{s}{2}} \, K_x \rangle_{L^2 (\partial \Omega)}.
\end{equation}
%
Gathering \eqref{Cani1} and \eqref{Cani2} we deduce that for any $u$ in  $\mathcal H_s (\Omega)$, 
\begin{equation*} 
  \Big\langle M^{\frac{s}{2}} \, \gamma_D (u) , \, M^{-\frac{s}{2}} \,\gamma_N \big( G_1 (x,\cdot)\big)  - M^{\frac{s}{2}} \, K_x \Big\rangle_{L^2 (\partial \Omega)}=0.
\end{equation*}
 Then by a density argument on $u$, we conclude that 
 $$ M^{\frac{s}{2}} \, K_x = M^{-\frac{s}{2}} \,\gamma_N \big( G_1 (x,\cdot)\big)   ,$$
 almost everywhere on $\partial\Omega$.
 Therefore,   for any  $x$ and $y$ in $\Omega$,
 \begin{align*}
     K_s(x,y) &= 
      \Big\langle M^{\frac{s}{2}} \, K_x, \,M^{\frac{s}{2}} \, K_y  \Big\rangle _{L^2 (\partial \Omega) } 
 \\ &=\Big\langle M^{-\frac{s}{2}} \, \gamma_N \big( G_1 (x,\cdot)\big) ,\, M^{-\frac{s}{2}} \, \gamma_N \big( G_1 (y,\cdot)\big) \Big\rangle _{L^2 (\partial \Omega) },
 \end{align*}
 which is the desired conclusion.
 \end{proof}

\subsection{Resemblance with the Hadamard variation formula}
An interesting observation in \cite{ELPL} is the resemblance of \eqref{l-formu} with the Hadamard variation formula 
for the Green function of the classical Laplacian, which we now recall. 
This formula expresses the derivative of the Green function $G_1$ with respect to the domain. 
More precisely, one considers some perturbations $\Omega_t$ of the domain $\Omega$ in the normal direction by modifying the  boundary $\partial\Omega$ of $\Omega$ in 
$$\partial\Omega_t=\Big\{y=x+t \alpha(x)\nu(x), \;x\in\partial\Omega\Big\},$$
where $\alpha\in C^\infty(\partial\Omega)$, and $t$ runs in a open interval containing $0$. 
Let us rename $G_\Omega$ the Green function
 $G_1$ associated with the domain $\Omega$, as defined by \eqref{DIR}, to emphasize here the dependence on the domain. Then we define the derivative of $G_\Omega$ with respect to the domain in  $\alpha$ as 
\begin{align}
     DG_\Omega (\alpha) (x,y):=\lim_{t\to 0, t \neq 0} \, \frac{G_{\Omega_t}(x,y)-G_\Omega(x,y)}{t} .
\end{align}
Then the pioneering discovery by Hadamard, see \cite{Hadamard}, is that not only does this limit exist, but it is even given by the following explicit integral formula: for any $x$ and $y$ in $\Omega$, 
\begin{align} \label{Hadamard}
    DG_\Omega (\alpha) (x,y) =\int_{\partial\Omega} \gamma_N \big(G_1(x,\cdot)\big)(z) \, \gamma_N \big(G_1(y,\cdot)\big)(z) \,  \alpha(z)d\sigma(z) .
\end{align}
Let us mention that a rigorous proof of Hadamard's formula was given later on by  Garabedian who also considered more general perturbations. After these pioneering works, shape derivative computations have received significant interest and several extensions of Hadamard's formula for more general elliptic boundary value problems were obtained, with various techniques. We refer here to \cite{HenrotPierre,Sokolowski} for more. 
\medskip

Following \cite{ELPL}, 
 it is now clear that  \eqref{l-formu} and \eqref{Hadamard} look very similar, in particular in the case where $s=0$ for which the only difference is the presence of the extra factor $\alpha$ in \eqref{Hadamard}.
 In \cite{ELPL}, it is said furthermore that the authors  "are of the opinion that the connection between Hadamard variation formula and the reproducing kernel just indicated, in the classical harmonic case, seems to be an isolated phenomenon peculiar to the second order case." 
%

\section{RKHS of fractional harmonic functions}

In this work, we address the issue of representing the kernel of some reproducing kernel Hilbert spaces associated with the fractional Laplace operators by some boundary integrals involving the fractional Green function. This formula will turn out to be  very similar to the Hadamard variation formula for  
the fractional Green function, invalidating the conjecture made in  \cite{ELPL}. 

\subsection{Fractional Laplace operators and their Hadamard variational formula}

One quite explicit formula for the  fractional Laplace operator of order $a$, with $a\in (0,1)$, in $N\in \N^*$ dimensions, reads 
$$
(-\Delta)^a u (z) := c_{N,a}\,  \pv \int_{\R^N}\frac{u(z)-u(y)}{|z-y|^{N+2a}}dy,
$$
with 
$$c_{N,a} :=\pi^{N/2}a4^a\frac{\Gamma(\frac{N+2a}{2})}{\Gamma(1-a)},$$
where $\Gamma$ is the gamma function, and where $\pv$ refers to the Cauchy principal value. 
Above $u$ is a function from $\R^N$ with real values. 
The reason for the presence of the normalization constant $c_{N,a}$
 is to match with the Fourier definition  which sets the fractional Laplace operator $(-\Delta)^a$ as the Fourier multiplier of symbol 
$|\xi|^{2a}$. 
 The Green function $G_a(x,y)$ associated to the operator $(-\Delta)^a$ and the homogeneous Dirichlet condition 
is the solution to 
\begin{equation} \label{DIRa}
    (-\Delta)^a G_a(x,\cdot)=\delta_x \quad\text{in}\quad \mathcal D'(\Omega)\qquad\&\qquad G_a(x,\cdot)=0\quad\text{in}\quad \R^N\setminus\Omega.
\end{equation}
Such a Green function plays the same role for the operator $(-\Delta)^a$ as the Green function $G_1$ for the classical Laplace operator. 
In the companion paper \cite{sidy-franck_1}, we computed the following Hadamard-type variational formula for the fractional Green function $G_a(x,y)$:  for any $x$ and $y$ in $\Omega$, 
\begin{align}\label{var-green}
   DG_a (\alpha) (x,y)  = \Gamma^2(1+a)\int_{\partial\Omega} \, \big(\gamma_0^a(G_a(x,\cdot))\big)(z)  \, \big(\gamma^a_0(G_a(y,\cdot))\big)(z)
  \alpha(z)  \;d\sigma(z).
\end{align}
Above, $ \Gamma^2(1+a)$ denotes the square of $ \Gamma(1+a)$ while $\gamma_0^a$ denotes the fractional trace of order $a$, that is 
$$\gamma_0^a(G_a(x,\cdot):= \gamma_D \big(\frac{G_a(x,\cdot)}{d^a}\big),$$ where $d$ is the distance function to the boundary $\partial\Omega$. 
The right-hand side of \eqref{var-green} makes sense because by boundary regularity, for any $x$ in  $\Omega$, $\gamma_0^a(G_a(x,\cdot))$ is $C^\infty$ on ${\partial\Omega}$, see \cite[Equation (2.14)]{Grubb-2014}. 
Let us point out that the analysis performed in \cite{sidy-franck_1} allows for more general perturbations. 
We also refer to the earlier works \cite{DalibardVaret,DFW} on the subject. 

\subsection{Toward some Sobolev spaces of $a$-harmonic functions}

The counterpart of the spaces  $\mathcal H_s(\Omega) $ of  Theorem \ref{Lions-theo}, which we are going to use for  the  fractional Laplace operator of order $a$, are some Sobolev spaces that have been introduced in the theory  of nonhomogeneous Dirichlet-type problems for  pseudodifferential operators   with non-integer orders. Indeed, there is a well-known calculus initiated by Boutet de Monvel
for integer-order  pseudodifferential operators with the 0-transmission property (preserving the $C^\infty$ regularity up to the boundary), including boundary value problems for elliptic differential operators and their inverses. 
Theories for operators without the transmission property have subsequently also been developed, see \cite{Eskin}.
A intermediate category of  pseudodifferential operators   are the ones satisfying a so-called $\mu$-transmission property, where $\mu$ can be any complex number. 
This has been initiated by H\"ormander in \cite{H65} and further developed by Grubb in \cite{Grubb2023,Grubb-2020,Grubb-2015,Grubb-2017,Grubb-2021,Grubb-2016,Grubb-2014,Grubb-2021-en}. This category includes the fractional Laplacian  $(-\Delta)^a$, for which 
the  $\mu$-transmission property is satisfied with $\mu = a$, and this  theory has allowed to obtain some solvability results for nonhomogeneous Dirichlet-type problems in large scales of Sobolev spaces. 
These spaces are the $a$-transmission Sobolev spaces  $H^{a(t)}(\comega)$, where $a,t \in \mathbb R$. 

\subsection{The $a$-transmission Sobolev spaces}

We recall here a few facts on the $a$-transmission Sobolev spaces and their roles in the solvability of  nonhomogeneous Dirichlet problem for the fractional Laplacian operator. We refer here to the series of papers \cite{Grubb2023,Grubb-2020,Grubb-2015,Grubb-2017,Grubb-2021,Grubb-2016,Grubb-2014,Grubb-2021-en}
by Grubb for this material.
We use below the  Fourier transformation with the following normalization: 
$$\mathcal F
u(\xi )= \hat u (\xi )=
\int_{{\mathbb R}^N}e^{-ix\cdot \xi }u(x)\, dx.
$$
  It is invertible from the space
$L^2(\mathbb R^N)$ onto $L^2(\mathbb R^N)$, and the inverse operator is 
$$(\mathcal F^{-1}v)(x)=
(2\pi )^{-N} \int_{{\mathbb R}^N}e^{+ix\cdot \xi }v(\xi )\, d\xi .$$

The classical Sobolev spaces over $\rn$ are defined by 
$$H^s(\rn)=\{u\in \mathcal S'(\mathbb R^N)\mid \langle{\xi }\rangle^s\hat
u\in L ^2(\mathbb R^N)\},\text{ with norm }\|\mathcal F^{-1}(\ang{\xi }^s\hat u)\|_{L^2}, $$
where  $\ang\xi
=(1+|\xi |^2)^{\frac12}$. 
 Next, we associate with $\Omega$ two types of Sobolev spaces (the notations originate from 
H\"ormander's books).
 \begin{Definition}
   For any $s\in\mathbb R$, we define  the {\it
  restricted} Sobolev space of order $s$: 
$$
  \overline H ^s(\Omega  )=r^+H^s (\mathbb R^N),$$
where $r^+$ denotes restriction to $\Omega $,  and the {\it supported} Sobolev space of order $s$: 
  $$ 
  \dot H^s (\comega )=\{u\in H^s (\mathbb R^N)\mid
\operatorname{supp}u\subset \comega  \}.
$$
    \end{Definition}
Let us state a few elementary properties of the 
  restricted and   supported Sobolev spaces. 
    \begin{Proposition}
  We have the following properties. 
    \begin{itemize}
\item  For all $s \in \R$, $\overline H^s(\Omega )$ and $\dot H^{-s}(\comega)$ are dual
spaces, with a duality bracket consistent with the $L^2$-scalar product.
\item The space $C_0^\infty (\Omega )$ of smooth functions with
compact support in $\Omega $ is dense in $\dot H^s(\comega)$ for all $s \in \R$. 
\item The space $\overline H^s(\Omega )$ coincides with
$\dot H^{s}(\comega)$ when $|s|<\frac12$.
  \item  For any $0<a<1$, the operator $r^+  (-\Delta )^a$ continuously maps $\dot H^{a}(\comega)$ into 
 $\overline H^{-a}(\Omega )$. 
  \end{itemize}
 \end{Proposition}

Above restriction and support have to be understood in the sense of the theory of distributions.

To study the nonhomogeneous Dirichlet problem for the fractional Laplacian operator 
  $(-\Delta )^a$,  some other Sobolev-type spaces are going to be introduced. 
 We first consider the case where the operator acts on functions defined  in a half-space. 
 The following auxiliary pseudo-differential operators are instrumental. 
   \begin{Definition}
     For any $t\in\mathbb R$ we define the {\it order-reducing operators}: 
$$
\Xi
_\pm^t=\operatorname{Op}((\ang{\xi '}\pm i\xi_N)^t),
\,  \text{ where } \,  
\ang{\xi '} := (1+|\xi' |^2)^{\frac12},
$$
with $(\xi ' ,\xi_N)$ being the splitting of $\xi$ in $\mathbb R^{N-1} \times \mathbb R$.
    \end{Definition}
 Below we give two  elementary properties of the order-reducing operators which help to understand their roles. 
    \begin{Proposition}
 For all $s,t \in \mathbb R$, the pseudo-differential operator 
$\Xi _\pm^t$ is a linear continuous mapping from 
 $H^{s}(\rn)$ onto $H^{s-t}(\rn)$  with
  inverse $\Xi _\pm^{-t}$.
Moreover,  $\Xi ^t_+$  preserves support in
   $\crnp$
   and is a linear continuous mapping from  
$ \dot H^{s}(\crnp)$ onto $\dot H^{s-t}(\crnp)$ with inverse 
$\Xi _+^{-t}$. 
    \end{Proposition}
Let us mention that $(\ang{\xi '}\pm i\xi_N)^t)$ are actually not classical symbols. 
There is a more refined choice of the 
so-called order-reducing operators, corresponding to some classical symbols, and which enjoy the same properties. 
We refer here to \cite{Grubb2023} and the reference therein to more on this issue.

We can now define  the $a$-transmission Sobolev spaces in the half-space $\rnp$.
      \begin{Definition}
    Let $a$ in $(0,1)$. For all $t \in \mathbb R$ with $t-a> - \frac12$, we define the  $a$-transmission Sobolev  spaces
 $$
H^{a(t)}(\crnp):= \Xi _+^{-a}e^+\overline
H^{t-a}(\rnp) ,
$$
where $e^+$ indicates the extension by zero from  $\rnp$ to $\rn$. 
   \end{Definition}
 %
For a more general domain, such as $\Omega $, the $a$-transmission Sobolev spaces are defined from the half-space case by using local coordinates as follows. 
  \begin{Definition} \label{transmibor}
      For  all $a$ in $(0,1)$,  
   for    $t >   a -1/2$, 
  the $a$-transmission Sobolev space  $H^{a(t)}(\comega)$
  is  defined over $\Omega $ by localization in such a way that 
 $u$ is in $H^t$ on
compact subsets of $\Omega $, and for every $x_0$ in $\partial\Omega $ has an open neighborhood $U$ and a $C^{\infty }$-diffeomorphism in $\rn$ mapping $U'$
to $U$ such that $U'\cap \crnp$ is mapped to $U\cap \comega$, and $u$ is pulled back to a function $u'$  in
$H^{a(t)}(\crnp)$ locally (i.e., $\varphi u'\in H^{a(t)}(\crnp)$ when $\varphi
\in C_0^\infty (U')$).
  \end{Definition}
Another presentation of these spaces is given in \cite{Grubb-2015}.  
Let us state a few elementary properties of  the $a$-transmission Sobolev spaces for some particular values of $t$ and $a$, for which we point out 
 \cite[Formula (13)]{Grubb-2020} 
  \cite[Formula (2.12)]{Grubb-2017}
  and
   \cite[ Lemma 2.2]{Grubb-2021}.
    \begin{Proposition}  Let $a$ in $(0,1)$ and $t \in \mathbb R$, we have the following properties.
    \begin{itemize}
\item  For $-\frac12<t-a<\frac12$,
$$
H^{a(t)}(\comega)=\dot H^t(\comega) .$$
\item  For $t-a
  >\tfrac1{2}$, 
 $$
H^{a(t)}(\comega) \subset \dot H^{t(-\varepsilon )}
  (\overline\Omega )+d^ae^+\overline H^{t-a}(\Omega ),
  $$
  where  $\varepsilon > 0$  if $t-a-\frac12$ is an integer or zero otherwise.  
  \item For any $0<a<1$, for any $s>-a-\frac{1}{2}$,
 the operator $r^+  (-\Delta )^a$ continuously maps the $a$-transmission Sobolev space $H^{(a)(s+2a)}(\comega) $ onto 
 $\overline H^{s}(\Omega )$. 
  \end{itemize}
\end{Proposition}

 We have the following result about fractional traces, which are weighted trace mappings with a weight corresponding to a power of the distance to the boundary. It follows from a direct application of \cite[Theorem 5.1]{Grubb-2015} with parameters $\mu = a-1$, $M=1$, and $p=2$.
\begin{thm}\cite[Theorem 5.1]{Grubb-2015} \label{trace-theo}
   For  all $a$ in $(0,1)$,  
   for    $t >   a -1/2$, 
 the fractional trace map $\gamma _0^{a-1}u=\gamma_D(u/d^{a-1})$ extends into a continuous linear map 
$$\gamma _0^{a-1} \colon  H^{(a-1) (t)}(\comega )\to
H^{t-a+\frac1{2}}(\partial \Omega ),$$
which is surjective, with kernel  $H^{a(t)}(\comega )$. 
\end{thm}
Above, $\gamma_D$ denotes the classical Dirichlet trace on the boundary $\partial \Omega$, see \eqref{gammad}.
We have the following existence and uniqueness result for the nonhomogeneous Dirichlet conditions, for which we refer to \cite[Equation (2.26)]{Grubb-2015}.

 \begin{thm}
 \label{f}
 Let $a$ in $(0,1)$.
 For any $\varphi \in
H^{s+a+1/2}(\partial\Omega )$, 
for all $s>-a-\frac{1}{2}$, 
 there is a unique solution $u\in H^{(a-1)(s+2a)}(\comega)$  of the nonhomogeneous Dirichlet problem
\begin{equation}
    \label{inho}
    (-\Delta )^au= 0 \text{ in }\Omega , \quad 
 \gamma _0^{a-1}u=\varphi \text{ on }\partial\Omega   \quad   \text{ and }  \quad 
\operatorname{supp}u\subset \comega .
\end{equation}
\end{thm}
Observe that we have the maximal  regularity authorized by the trace theorem, since $s+2a$ plays here the role of $t=s+2a$ in Theorem \ref{trace-theo}. 

\begin{Definition} \label{def-Psa}
Let $a$ in $(0,1)$, $s  > -a -1/2$ and 
\begin{equation} \label{theta}
     \theta := {\frac{s}{2}+\frac{a}{2}+\frac1{4}}.
\end{equation} 
We define the mapping $\mathcal P_{a,s}$ which maps  any $\varphi$ in $H^{2\theta}(\partial \Omega )$ to 
 $ \mathcal P_{a,s}[\varphi] := u$ in $H^{(a-1)(s+2a)}(\comega)$

 the unique solution of \eqref{inho}.
 \end{Definition}
It follows from Theorem \ref{trace-theo} and Theorem  \ref{f} that, for any  $a\in (0,1)$ and for any  $s>-a-\frac{1}{2}$, the mapping 
 $\mathcal P_{a,s}$ is a one-to-one mapping from  $H^{2\theta}(\partial \Omega )$ with $\theta$ given by \eqref{theta},  into $H^{(a-1)(s+2a)}(\comega)$
 whose inverse is the trace mapping 
$ \gamma _0^{a-1}$.

\subsection{Some Hilbert spaces of $a$-harmonic functions with Sobolev traces and a fractional Poisson formula}
\label{sec_poissona}

We are now ready to define some appropriate Sobolev spaces $\mathcal H_{a,s}(\Omega)$ of $a$-harmonic functions that will  
play the same role as the spaces $\mathcal H_{s}(\Omega)$ played for the classical Laplace operator. 
Recall that in Definition  \ref{def-Hs}, we defined 
 $\mathcal H_s(\Omega) $ as the space of harmonic distributions $u$ in $ \Omega$ with traces in $H^s(\partial \Omega)$, see also \eqref{inpr}.
Similarly, the mappings $\mathcal P_{a,s}$ allow us to
define  some spaces $\mathcal H_{a,s}(\Omega)$ of $a$-harmonic functions as follows and to transport the Hilbert structure from $H^{2\theta}(\partial \Omega )$ 
to $\mathcal H_{a,s}(\Omega) $. 
\begin{Definition} \label{def-Hsa}
Let $a$ in $(0,1)$, $s  > -a -1/2$, and $ \theta $ be given by \eqref{theta}.
 We define  $\mathcal H_{a,s}(\Omega)$  as the Hilbert space of the functions $u$ in the $a$-transmission Sobolev space
 $H^{(a-1)(s+2a)}(\comega)$ such that 
 $(-\Delta )^au= 0 $ in $\Omega $, 
and $\operatorname{supp}u\subset \comega$, 
 equipped with the scalar product of $H^{2\theta}(\partial \Omega )$ for the trace $\gamma _0^{a-1} (u)$, that is 
 \begin{equation} \label{sca}
     \big\langle u,\,v\big\rangle_{\mathcal H_{a,s}(\Omega) }  :=
     \Big\langle  \gamma _0^{a-1}(u),\,  \gamma _0^{a-1}(v) \Big\rangle_{2\theta}.
 \end{equation}
 \end{Definition}
With this definition, the mapping  $\mathcal P_{a,s}$ is a one-to-one isometry  from  $H^{2\theta}(\partial \Omega )$, with $\theta$ given by \eqref{theta}, onto $\mathcal H_{a,s}(\Omega)$, and satisfies both  $\gamma_0^{a-1}\circ P_{a,s} = \mathrm{Id}$ on $H^{2\theta}(\partial\Omega)$,
and $P_{a,s}\circ \gamma_0^{a-1} = \mathrm{Id}$ on $\mathcal H_{a,s}(\Omega)$.


%

%

The next result provides an explicit representation formula for functions in  $\mathcal H_{a,s}(\Omega)$. 

 %
  \begin{Lemma} \label{Poisson-f}
  Let  $a\in (0,1)$, $s>-a-\frac{1}{2}$, and
 $\theta$ be given by \eqref{theta}. 
Let 
  $x\in \Omega$ 
and $u\in \mathcal{H}_{a,s}(\Omega)$.
Then
 \begin{align} \label{main-id-s}
u(x) &=  \Gamma (a)\Gamma(a+1)
 \Big\langle  \gamma _0^{a-1} (u),\, M^{-2\theta} \gamma _0^{a} (G_a (x,\cdot))  \Big\rangle _{2\theta}.
\end{align}
\end{Lemma}
%


\begin{proof}
Let $u$ in $\mathcal H_{a,s}(\Omega)$. Recalling that $\gamma^{a-1}_0(u)$ is in $H^{2\theta}(\partial\Omega)$,
by density, we have a sequence $(g_n)_{n\in\mathbb N}$ in $C^\infty(\partial\Omega)$ such that
\begin{equation} \label{ref-jeudi}
g_n \to \gamma^{a-1}_0(u) \quad \text{in } H^{2\theta}(\partial\Omega).
\end{equation}
By Theorem \ref{f}, for any $n$ in $\mathbb N$, we set 
 $u_n := \mathcal P_{a,s}[g_n]$. This 
is the unique $u_n$ in $H^{(a-1)(s+2a)}(\Omega)$
solution of
\begin{equation*}
(-\Delta)^a u_n = 0 \ \text{in }\Omega \ \ \& \ \ \gamma^{a-1}_0(u_n) = g_n \ \text{on }\partial\Omega.
\end{equation*}
We fix $x$ in $\Omega$. Let $\rho$ in $C^\infty_0(\mathbb R^N)$ with $\rho \ge 0$ and
$\int_{\mathbb R^N}\rho = 1$, and set $\rho_\varepsilon(z) := \varepsilon^{-N}\rho(z/\varepsilon)$.
For $\varepsilon>0$ small enough so that $\mathrm{supp}\,\rho_\varepsilon(\cdot-x)\subset \Omega$,
we define
\begin{equation} \label{*}
v_{x,\varepsilon}(z) := \int_\Omega G_a(z,y)\rho_\varepsilon(y-x)\,dy.
\end{equation}
Then $v_{x,\varepsilon}=0$ in $\mathbb R^N\setminus\Omega$, and in $\mathcal D'(\Omega)$ we have
\begin{equation} \label{**}
(-\Delta)^a v_{x,\varepsilon} = \rho_\varepsilon(\cdot-x).
\end{equation}
Moreover, since $\rho_\varepsilon(\cdot-x)$ is smooth with compact support in $\Omega$,
the function $v_{x,\varepsilon}$ is regular enough for the hypotheses of \cite[Theorem 5.1]{Grubb-2020}.
We apply Grubb's exact Green formula to the pair $(u_n,v_{x,\varepsilon})$.
Since $(-\Delta)^a u_n = 0$ in $\Omega$ and $(-\Delta)^a v_{x,\varepsilon}=\rho_\varepsilon(\cdot-x)$ in $\Omega$,
we obtain
\begin{equation} \label{***}
\int_\Omega u_n(y)\rho_\varepsilon(y-x)\,dy
= \Gamma(a)\Gamma(a+1)\int_{\partial\Omega} g_n(z)\,\gamma_0^a(v_{x,\varepsilon})(z)\,d\sigma(z).
\end{equation}

We now let $\varepsilon\to0$. Since $g_n$ in $C^\infty(\partial\Omega)$, the solution $u_n$ is smooth in $\Omega$,
hence $\int_\Omega u_n(y)\rho_\varepsilon(y-x)\,dy \to u_n(x)$.
On the other hand, for each $z$ in  $\partial\Omega$, we have by definition

\[
\gamma_0^a(v_{x,\varepsilon})(z) = \lim_{t\to 0^+}\frac{v_{x,\varepsilon}(z+t\nu(z))}{t^s}=\lim_{t\to 0^+}\int_{\Omega}\frac{G_a(z+t\nu(z),y)}{t^s}\rho_\epsilon(y-x)dy.
\]
Then, using the standard estimate on the Green function (see e.g \cite{Tedeuz}), we get for $t>0$ sufficiently small that $t^{-s}G_a(z+t\nu(z),y)\leq C|y-z|^{s-N}\in L^1(\Omega)$. Since $\rho_\varepsilon(\cdot-x)$ is bounded, it therefore follows by the Lebesgue dominated convergence theorem that
\[
\gamma_0^a(v_{x,\varepsilon})(z)
= \int_\Omega \rho_\varepsilon(y-x)\,\gamma_0^a(G_a(y,\cdot))(z)\,dy.
\]
Next, by the smooth dependence of
$\gamma_0^a(G_a(y,\cdot))$ on $y$ in compact subsets of $\Omega$ (in particular \\$y\mapsto \gamma_0^a(G_a(y,\cdot))(z)$
is smooth for each $z$ in $\partial\Omega$), we have for each $z$ in $\partial\Omega$
\begin{equation} \label{****}
\gamma_0^a(v_{x,\varepsilon})(z)
= \int_\Omega \rho_\varepsilon(y-x)\,\gamma_0^a(G_a(y,\cdot))(z)\,dy
\rightarrow \gamma_0^a(G_a(x,\cdot))(z)\quad\text{as $\varepsilon\to 0^+$}.
\end{equation}
Therefore, passing to the limit in \eqref{***}, we obtain
\begin{equation}
u_n(x) = \Gamma(a)\Gamma(a+1)\int_{\partial\Omega} g_n(z)\,\gamma_0^a(G_a(x,\cdot))(z)\,d\sigma(z).
\end{equation}
This regularization argument is similar to the one used in \cite{Sidy-Franck-BP} to justify Green-type identities
in the presence of an interior singularity.

Using Definition \ref{def1.3} of the Sobolev scalar product on $\partial\Omega$, we  arrive at:
\begin{align*}
u_n(x) &= \Gamma(a)\Gamma(a + 1)\Big\langle
M^\theta g_n, \,M^{-\theta}\gamma^a_0(G_a(x, \cdot))\Big\rangle_{L^2(\partial\Omega)}\\
&=\Gamma(a)\Gamma(a + 1)\Big\langle
g_n,\, M^{-2\theta}\gamma^a_0(G_a(x, \cdot))\Big\rangle_{2\theta}.
\end{align*}

For all $x$ in $\Omega$, by \eqref{ref-jeudi}, $u_n(x)$ converges to $w(x)$, where the function $w:\Omega\to\mathbb R$
is defined by
\begin{equation}
w(x) =\Gamma(a)\Gamma(a + 1)\Big\langle
\gamma^{a-1}_0(u),\, M^{-2\theta}\gamma^a_0(G_a(x, \cdot))\Big\rangle_{2\theta}
\end{equation}

On the other hand, the sequence $(u_n)$ converges in $H_{a,s}(\Omega)$ to $u$.
By the interior elliptic regularity of the pseudo-differential operator $(-\Delta)^a$,
this entails that for all $x$ in $\Omega$, $u_n(x)$ converges to $u(x)$.
By the uniqueness of the limit, this yields the desired conclusion that $u(x)=w(x)$.
\end{proof}

Let us point out that, as a corollary, we deduce that for any $\varphi$ in $H^{(a-1)(s+2a)}(\comega)$, for any $x$ in $\Omega$, 
\begin{equation*} 
 \mathcal P_{a,s}[\varphi] (x) =   
 \Gamma (a)\Gamma(a+1)
  \Big\langle  \varphi,\, M^{-2\theta} \gamma _0^{a} (G_a (x,\cdot))  \Big\rangle _{2\theta}.
\end{equation*}

\subsection{A fractional counterpart of Lions' formula}

 Our main result establishes a counterpart of Theorem \ref{Lions-theo} in the case of $a$-harmonic functions by showing that, for an appropriate range of indexes, the spaces $\mathcal H_{a,s}(\Omega)$ are RKHS with two-point kernels given by  a boundary integral of some derivatives of the Green function  $G_a(x,\cdot)$  of the fractional  Laplacian with a singularity at $x$.
We recall that $M$ is the  smoothed Laplace-Beltrami operator on $\partial \Omega$, see Definition \ref{DefBel}.
\begin{thm} \label{s-Lions-theo}
Let  $a\in (0,1)$, $s>-a-\frac{1}{2}$, and
 $\theta$ be given by \eqref{theta}. 
Then the space $\mathcal H_{a,s}(\Omega)$ is a RKHS, and its  two-point kernel, which we  denote $K_{a,s}$, is given,  for any  $x$ and $y$ in $\Omega$,  by 
 \begin{equation} \label{l-formu-s}
 K_{a,s}(x,y)=  \Gamma^2(a) \Gamma^2(a+1)  \int_{\partial \Omega} \Big( M^{-\theta} \,  
  \gamma_0^{a}(G_a(x,\cdot))
   \Big) \Big(M^{-\theta} \,  \gamma_0^{a}(G_a(y,\cdot))  \Big) d\sigma.
 \end{equation}
\end{thm}

Therefore, the formula \eqref{l-formu-s} is  very similar to the Hadamard variation formula \eqref{var-green} for the Green function $G_a$, proving that the connection in \cite{ELPL} 
 between the classical Hadamard variation formula and the reproducing kernel associated with the classical Laplace operator, and for a large class of second order elliptic systems, also holds true for some elliptic operators of fractional order $2a$, invalidating the conjecture made in  \cite{ELPL}. 

\begin{remark}
One may  easily check that 
 the formula  \eqref{l-formu-s} is consistent with the classical Lions formula  \eqref{l-formu} for the Laplace operator; since, as $a \rightarrow 1^-$, 
  for any pair of $x$ and $y$ in $\Omega$, for any $s > - \frac32$, 
  \begin{equation*} 
 K_{a,s}(x,y)  \rightarrow   K_{s+3/2} (x,y) .
 \end{equation*}
 which corresponds to the fact that the space $\mathcal H_s (\Omega)$ corresponds to the scalar product in  the Sobolev space $H^s(\partial \Omega)$, whereas the 
  space $\mathcal H_{1,s}(\Omega)$, obtained from Definition \ref{def-Hsa}  by substituting  $a=1$, 
 corresponds to the scalar product in  the Sobolev space $H^{s+3/2}(\partial \Omega)$.
\end{remark}


\begin{remark}
Let us also mention that the formula \eqref{l-formu-s} could be useful 
to determine the boundary behavior of the reproducing kernel, thanks to the knowledge of  the boundary behavior of the Green function $G_a$. 
\end{remark}

\begin{proof}
Let $x$ be in $\Omega$.
As a consequence of Lemma \ref{Poisson-f} and the Cauchy-Schwarz inequality, 
the  evaluation mapping $E_x$ from $\mathcal{H}_{a,s}(\Omega)$ to $\R$ that  linearly maps $u$ to $E_x (u) = u(x)$ is continuous. 
Since $\mathcal H_{a,s}(\Omega)$ is a Hilbert space, 
 it follows from Riesz' theorem that there exists $K_x$ in 
 $\mathcal H_{a,s}(\Omega)$
such that for any $u$ in   $\mathcal H_{a,s}(\Omega)$,
 \begin{align*}
 u(x) &=  \Big\langle \gamma _0^{a-1} (u),\, \gamma _0^{a-1} (K_x)\Big\rangle _{2\theta}
\\
 &= 
 \Big\langle M^{\theta} \gamma _0^{a-1}(u), M^{\theta}  \gamma _0^{a-1} (K_x) \Big\rangle_{L^2(\partial \Omega )}.
\end{align*}
 Then, proceeding as in the proof of 
Theorem \ref{Lions-theo}, 
 by identification with \eqref{main-id-s} and by using a density argument, 
 we conclude that 
$$ M^\theta \,\gamma_0^{a-1}(K_x) = \Gamma (a)\Gamma
(a+1)   M^{-\theta} \, \,  \gamma_0^{a}(G_a(x,\cdot)) \quad \text{a.e.} \, \text{on}\, \partial\Omega. $$
Then, for any  $x$ and $y$ in $\Omega$, 
 \begin{align*}
  K_{a,s} (x,y)   &=\Big\langle \gamma_0^{a-1}(K_x),\,\gamma_0^{a-1}(K_y) \Big\rangle _{2\theta}
  \\ &=  \Big\langle M^\theta \, \gamma_0^{a-1}(K_x),\,M^\theta \, \gamma_0^{a-1}(K_y) \Big\rangle _{L^{2}(\partial \Omega )} 
 \\ &=\Gamma^2 (a)\Gamma^2
(a+1)  \Big\langle M^{-\theta} \,  \gamma_0^{a}(G_a(x,\cdot)) ,\, M^{-\theta} \,  \gamma_0^{a}(G_a(y,\cdot)) \Big\rangle_{L^2 (\partial \Omega) }.
 \end{align*}
 This provides the formula \eqref{l-formu-s} and ends the proof of 
 Theorem \ref{s-Lions-theo}. 
\end{proof}

\section{RKHS associated with the steady Stokes problem}

As a complementary remark regarding the similarity, or lack thereof, of Lions' type formula for the two-point kernel of  RKHS with the Hadamard variation formula, let us highlight that there are some second order systems for which the two formulas do not appear quite the same. Indeed, let us restrict ourselves here to the case where $N=3$, so that  $\Omega$ is a $C^\infty$ bounded connected open set in $\R^3$.
For any $s$ in $\R$,  consider the space $\mathfrak H_s$ of the traces on $\partial \Omega$ of solutions to the steady Stokes problem:
\begin{equation} \label{eq_stoke}
\left\{
\begin{array}{rcl}
- \Delta u + \nabla p &=& 0 \, \\
\operatorname{div} u &= & 0 \, 
\end{array}
\right.
\quad \text{ in $\Omega$} ,
\end{equation}
endowed with the scalar product of the Sobolev space $H^s(\partial \Omega ; \R^3)$.

We  recall that  the system  \eqref{eq_stoke} aims to describe steady, incompressible fluids with a zero-Reynolds number. %
In particular, the second equation encodes the incompressibility of the fluid velocity field $u$, which is vector-valued, and this incompressibility constraint  translates into the presence  of the gradient of the fluid pressure in the first equation, the fluid pressure $p$ being itself scalar-valued.  We refer to \cite{Galdi}  for more.

We also recall the following classical Poisson formula: for any solution $(u,p)$ of the steady Stokes problem \eqref{eq_stoke} we have 
\begin{equation} \label{eq_Green-St}
u(x)
= \int_{\partial \Omega} ( \Sigma (\mathfrak G(x,\cdot),\mathfrak P(x,\cdot))  n) \cdot u \,d\sigma.
\end{equation}
where
\begin{equation} \label{eq_Newt2}
\Sigma(u,p) = 2 D(u) - p \mathbb{I}_3 \quad   \text{ where } 2 D(u):= \nabla u + (\nabla u)^T , 
\end{equation}
and $\mathbb{I}_3$ is the $3 \times 3$ identity matrix 
and $(\mathfrak G(x,\cdot),\mathfrak P(x,\cdot))$ is the Green function associated with 
 the Stokes system  
 that is, the unique tensor such that for any $b \in \R^3$, $(u,p) := (\mathfrak G(x,\cdot)b,\mathfrak P(x,\cdot)b)$ is  the unique solution to the problem: 
\begin{equation*} 
\left\{
\begin{array}{rcl}
- \Delta u + \nabla p &= b \delta_x &  \, \\
\operatorname{div} u &=  0  &\, 
\end{array}
\right. \quad \text{ in $\Omega$} , \end{equation*}
with \begin{equation*} 
 u =0\quad\text{on}\quad \partial\Omega.
\end{equation*}
 A reasoning similar to the ones above leads to the conclusion that the space $\mathfrak H_s$ is a RKHS with the following
 Lions type formula for its $2$-points kernel $\mathfrak K_s$: 
$$ \mathfrak K_s (x,y) =   \int_{\partial\Omega} \Big((M^{-\frac{s}{2}} \, \gamma_D \big(\Sigma(\mathfrak G(x,\cdot),\mathfrak P(x,\cdot))n\big)\Big)(z) \cdot \Big(M^{-\frac{s}{2}} \, \gamma_D \big(  \Sigma(\mathfrak G(y,\cdot),\mathfrak P(y,\cdot))n   \big) \Big)(z)\,d\sigma(z) .
$$
 On the other hand, it is known since some works by Simon, see for instance \cite{Simon} and the recent works by \cite{KU,Oz}, 
 that 
 the Hadamard variation formula in the case of the Stokes equations reads: for any $b \in \R^3$,
 $$D(\mathfrak G b) (\alpha) (x,y) =  
 \int_{\partial\Omega} \gamma_N \big(\mathfrak G(x,\cdot) b\big)(z) \, \gamma_N \big(\mathfrak G(y,\cdot)b\big)(z) \,  \alpha(z)d\sigma(z) , $$
where $\alpha\in C^\infty(\partial\Omega)$.


\end{document}